\documentclass[11pt]{amsart}
\usepackage{amsmath,amsxtra,amssymb,color,latexsym,epsfig,amscd,amsthm,subfigure,fancybox,epsfig}
\usepackage[mathscr]{eucal}
\usepackage{bbm}
\usepackage{float}
\usepackage{graphicx}
\usepackage{enumerate}
\usepackage{enumitem}
\usepackage{epsfig}
\usepackage{epstopdf}
\usepackage{cases}
\usepackage{hyperref}
\newcommand{\CN}{\mathcal{N}}

\hypersetup{
    colorlinks=true,
    linkcolor=blue,
    filecolor=magenta,
    urlcolor=cyan,
    citecolor=blue,
}

\newtheorem{thm}{Theorem}[section]

\newtheorem{lm}{Lemma}[section]
\newtheorem{prop}{Proposition}[section]

\newtheorem{cor}[thm]{Corollary}

\theoremstyle{definition}

\theoremstyle{remark}
\newtheorem{rem}{Remark}[section]

\numberwithin{equation}{section}

\DeclareMathOperator{\sgn}{sgn}
\newcommand{\eps}{\varepsilon}

\newcommand{\M}{\mathcal{M}}

\newcommand{\E}{\mathbb{E}}

\newcommand{\by}{\mathbf{y}}

\newcommand{\Lom}{\mathcal{L}}

\newcommand{\N}{{\mathbb{Z}}_+}
\newcommand{\Q}{{\mathbb{Q}}}

\newcommand{\BX}{\bar{X}}
\newcommand{\BS}{\bar{S}}
\newcommand{\BP}{\bar{P}}
\newcommand{\BY}{\mathbf{Y}}

\newcommand{\PP}{\mathbb{P}}

\newcommand{\R}{\mathbb{R}}

\numberwithin{equation}{section}
\newcommand{\1}{\boldsymbol{1}}

\newcommand{\wdt}{\widetilde}

\newcommand{\op}{{\mathcal L}}

\newcommand{\bed}{\begin{equation}}
\newcommand{\eed}{\end{equation}}
\newcommand{\bea}{\bed\begin{array}{rl}}
\newcommand{\eea}{\end{array}\eed}

\newcommand{\barray}{\begin{array}{ll}}
\newcommand{\earray}{\end{array}}
\newcommand{\diag}{{\rm diag}}
\newcommand{\dist}{{\rm dist}}

\def\disp{\displaystyle}

\def\bar{\overline}
\def\hat{\widehat}
\def\a.s{\text{\;a.s.\;}}

\def\bmu{\boldsymbol{\nu}}

\def\beq{\begin{equation}}
	
	\def\eeq{\end{equation}}

\def\ben{\begin{enumerate}}
	
	\def\een{\end{enumerate}}

\def\beqar{\begin{eqnarray}}
	
	\def\eeqar{\end{eqnarray}}

\def\beqarr{\begin{eqnarray*}}
	
	\def\eeqarr{\end{eqnarray*}}

\begin{document}
\title{Dynamics of stochastic microorganism flocculation models}

\author[A. Hening]{Alexandru Hening }
\address{Department of Mathematics\\
Texas A\&M University\\
Mailstop 3368\\
College Station, TX 77843-3368\\
United States
}
\email{ahening@tamu.edu}

\author[N. T. Hieu]{Nguyen T. Hieu}
\address{Faculty of Mathematics Mechanics and Informatics\\ Vietnam National University of Science, Hanoi
	334 Nguyen Trai St\\
	Hanoi,
	Vietnam}
\email{hieunguyentrong@gmail.com}

\author[D.H. Nguyen]{Dang H. Nguyen }
\address{Department of Mathematics \\
University of Alabama\\
345 Gordon Palmer Hall\\
Box 870350 \\
Tuscaloosa, AL 35487-0350 \\
United States}
\email{dangnh.maths@gmail.com}

\author[N. Nguyen]{ Nhu Nguyen }
\address{
	Department of Mathematics and Applied Mathematical Sciences\\
University of Rhode Island\\
5 Lippitt Road, Suite 200\\
Kingston, RI 02881-2018\\
United States}
\email{nhu.nguyen@uri.edu }

\maketitle

\begin{abstract}
In this paper we study the dynamics of stochastic microorganism flocculation models. Given the strong influence of environmental and seasonal fluctuations that are present in these models, we propose a stochastic model that includes multiple layers of stochasticity, from small Brownian fluctuations, to possibly large changes due to environmental `shifts'. We are able to give a full classification of the asymptotic behavior of these models. New techniques had to be developed to prove the persistence and extinction of the process as the system is not in Kolmogorov form and, as a result, the analysis is significantly more involved.
\bigskip

\end{abstract}

\section{Introduction}\label{sec:int}

Flocculation describes a process where particles coagulate, coalesce or agglomerate due to weak interactions. This leads to phase separation due to the formation of large clusters. A flocculant is an agent that can lead to flocculation and therefore to a coagulation of various dispersed particles into an aggregate - see \cite{salehizadeh2001extracellular, renault2009chitosan, salehizadeh2018recent, liu2021recent}. There are three main types of flocculants: organic, inorganic and biological. There are good reasons \cite{liu2021recent} to prefer bioflocculants as organic and inorganic flocculants can be unsafe due to toxicity. Among bioflocculants an interesting class is the one composed of microbial flocculants. These are not toxic and do not provide primary or seconday sources of pollution \cite{liu2015bioflocculant}. Microbial flocculants have already been used extensively in various industrial fields, from wastewater treatment \cite{agunbiade2017flocculating, pu2018isolation} and activated sludge dewatering \cite{liu2014optimized} to the removal of pathogens from water \cite{more2014extracellular, zhao2013characterization} and in the harvesting of marine microalgae \cite{lee2009microbial}.

Our goal was to explore some important mathematical models which describe the flocculation of microorganisms and to describe the long-term behavior of these systems. It turns out that the analysis is quite involved as these systems are not in the standard Kolmogorov form
\[
\frac{dX_i}{dt}(t) = X_i f_i(X(t))
\]
which usually appears in the population dynamics literature. 

The role of the environment in shaping species interactions cannot be overlooked. Since environmental conditions often vary in random and unpredictable ways, it is essential to develop mathematical models that can account for both stochastic environmental fluctuations and the nonlinear dynamics of intra- and interspecies interactions.

There are cases where biotic interactions alone may drive a species to extinction, yet the introduction of random environmental variation can reverse this outcome and promote coexistence. Conversely, deterministic systems in which all species persist may lead to extinction once environmental randomness is introduced. Recent advances in this field rely on modeling population dynamics through discrete- or continuous-time Markov processes and studying their long-term behavior (\cite{C00, ERSS13, EHS15, LES03, SLS09, SBA11, BEM07, BS09, BHS08, CM10, CCA09, FS24, B23}).

One useful approach to investigating species coexistence is to calculate the average per-capita growth rate of a population when it is rare. A positive value suggests that the population tends to increase from low density, while a negative value implies decline and eventual extinction. In two-species systems, coexistence occurs when each species is capable of invading while rare, provided the other species is at equilibrium (\cite{T77, CE89, EHS15}).

A well-established theoretical framework exists for coexistence in deterministic models (\cite{H81, H84, HJ89}). Beginning with the foundational work of \cite{H81}, it was demonstrated that persistence is guaranteed if there exists a set of fixed weights assigned to the populations such that the weighted sum of their invasion rates remains positive for any invariant measure supported on the boundary.

More recently, considerable attention has been given to to the study of stochastic differential equations (SDE) models models. For general stochastic differential equations with arbitrary noise intensity, sufficient conditions for persistence and extinction have been established in \cite{SBA11, HN16, HNC20, HNS21, B23, FS24}.

While SDE models are well studied, Brownian noise is not always the most suitable type of noise. Even though, one could argue, small Brownian fluctuations should always be included, sometimes there are more abrupt and significant environmental changes. These can be modeled by a discrete component $\xi(t)$ which switches randomly between finitely many environmental states. This leads to stochastic differential equations with switching. In a fixed environmental state the system is modeled by stochastic differential equations. This way we can capture the more realistic behavior of two types of environmental fluctuations:
\begin{itemize}
  \item major changes of the environment (daily or seasonal changes),
  \item small Brownian fluctuations within each environment.
\end{itemize}
General properties for these processes have been studied thoroughly \cite{YZ10, ZY09, nguyen2017certain}. However, there are few general results regarding the persistence of ecological systems modeled by stochastic differential equations with switching \cite{hening2021stationary}.

Another novelty in the model studied in this paper is that boundedness and dissipativity conditions fail. Loosely speaking, dissipativity requires that if a species has a high density, then the species will strongly drift towards zero. This fails in our model and therefore we cannot immediately use the methods from \cite{HN16, HNC20, HNS21}.

\subsection{Mathematical Setup and Results}\label{s:ma}

Consider first the model without any environmental fluctuations. 

\begin{equation}\label{main_det}
	\begin{aligned}
\frac{dS}{dt}(t) &= [D(S_0-S(t))-\mu_1S(t)X(t)]\\
\frac{dX}{dt}(t)&=\mu_2S(t)X(t)-DX(t)-h_2X(t)P(t)\\
\frac{dP}{dt}(t)&=D(P_0-P(t))-h_3X(t)P(t)
\end{aligned}
\end{equation}
where $S(t), X(t), P(t)$ are the concentrations of the nutrient (or the medium), microorganisms and flocculants at time $t\geq 0$. The parameters have the following biological interpretations: $D, S_0 >0$ are the velocity and the input concentration of the medium, $\mu_1>0$ is the consumption or death rate of the medium, $\mu_2>0$ is the yield of the microorganisms, $h_2>0$ is the flocculation rate, $P_0>0$ is the input concentration of flocculant, and $h_3>0$ is the consumption rate of the flocculant. 

The long-term behavior of \eqref{main_det} is known. According to \cite{zhang2020stationary, zhang2021asymptotic}, if $S_0< \frac{h_2P_0+D}{\mu_2}$ the microorganisms go extinct and the dynamics converges to $(S_0,0, P_0)$ while if $S_0> \frac{h_2P_0+D}{\mu_2}$ there exists a unique asymptotically stable fixed point $(S^*, X^*, P^*) \in \R_+^{3,\circ}:=(0,\infty)^3.$

We consider a stochastic model of the form
\begin{equation}\label{main}
	\begin{aligned}
dS(t) &= [D(\xi(t))(S_0-S(t))-\mu_1(\xi(t))S(t)X(t)]dt+\sigma_1 S(t)dW_1(t)\\
dX(t)&=[\mu_2(\xi(t))S(t)X(t)-D(\xi(t))X(t)-h_2(\xi(t))X(t)P(t)]dt+\sigma_2X(t)dW_2(t)\\
dP(t)&=[D(\xi(t))(P_0-P(t))-h_3(\xi(t))X(t)P(t)]dt+\sigma_3P(t)dW_3(t)
\end{aligned}
\end{equation}
where $(W_1(t), W_2(t), W_3(t))$ is a standard Brownian motion on $\R^3$. Define $\BY(t):=(S(t), X(t), P(t))$ and let $\by \in \R_{+}^{3,\circ}$ denote the initial conditions, that is $\BY(0):=(S(0), X(0), P(0))=\by = (s, x, p)=(y_1,y_2,y_3)$.

The switching process $\xi(t)$ lives on a finite state space $\M:=\{1,\dots,m_0\}$ and is defined by
\begin{equation}\label{e:tran}\begin{array}{ll}
&\disp \PP\{\xi(t+\Delta)=j~|~r(t)=i, \BY(s),r(s), s\leq t\}=q_{ij}\Delta+o(\Delta) \text{ if } i\ne j \
\hbox{ and }\\
&\disp \PP\{\xi(t+\Delta)=i~|~r(t)=i, \BY(s),r(s), s\leq t\}=1+q_{ii}\Delta+o(\Delta).\end{array}\end{equation}
The above equation says that the probability that the process $\xi(t)$ jumps from $\xi(t)=i$ to the state $j\neq i$ in the small time $\Delta$ is approximately equal to $q_{ij}\Delta$.

Here $q_{ii}:=-\sum_{j\ne i}q_{ij}$ quantifies the probability of staying in the same state $i$. We will assume that $Q=(q_{ij})_{m_0\times m_0}$ is irreducible. This implies that the process $\xi(t)$ has a unique stationary distribution $(\pi_1,\dots,\pi_{m_0})$ which is the unique solution of
\[
\sum_j (\pi_j q_{ji} - \pi_i q_{ij})=0.
\]

It is well-known \cite{YZ10} that a process $(\BY(t),\xi(t))$ satisfying \eqref{main} and \eqref{e:tran}
is a Markov process with generator acting on functions $G:\R_+^n\times\CN\mapsto\R$ that are continuously differentiable in $\by$ for each $k\in\CN$ as
\begin{equation}\label{e:gen}
\Lom G(\by, k)=\sum_{i=1}^n F_i(\by,k)\frac{\partial G}{\partial y_i}(\by,k)+\frac{1}{2}\sum_{l=1}^3\sigma_{l}^2\frac{\partial^2 G }{\partial y_l^2}(\by,k)+\sum_{l\in\M}q_{kl}G(\by,l).
\end{equation}
where $F_i(\by,k)$ are the drift terms from \eqref{main}.

The first theorem tells us that we can bound the moments of the process, and that the process stays in compact sets with a large probability.
\begin{thm}\label{thm1}
	For any initial value $(\by,k)\in\R^3_+\times\M$, there exists uniquely a global solution
	$\BY(t)$ to \eqref{main} 
	such that $\PP_{\by,k}\{\BY(t)\in\R^3_+,\  \forall t\geq0\}=1.$
	Moreover, $S(t)>0$ and $P(t)\geq 0$ for all $t>0$ with probability 1 and if $X(0)=0$ then $X(t)=0$  for all $t\geq0$ with probability 1. If $X(0)>0$ then $X(t)>0$ for all $t\geq 0$ with probability 1.
	We also have that  $(\BY(t), \xi(t))$ is a Markov-Feller process on $\R^3\times\M$.
	Furthermore,
there exists $q_0>1$ such that for any $q\in[1,q_0]$,
\begin{equation}\label{e1-thm1}
	\E_{\by,k} (1+|\BY(t)|)^{q}\leq (1+|\by|)^{q} e^{-c_{1,q} t} + c_{2,q} \text{ where }|\by|=s+x+p.
\end{equation}
for some positive constants $c_{1,q}$ and $c_{2,q}$.
There exists $\bar K>0$ such that
\begin{equation}\label{e2-thm1}
\E_{\by,k} (1+|\BY(t)|)^{2}\leq  e^{\bar K t} (1+|\by|)^2.
\end{equation}
For any $\eps>0, H>0, T>0$, there exists $\wdt K(\eps, H, T)>0$ such that
\begin{equation}\label{e3-thm1}
\PP_{\by,k}\left\{|\BY|\leq \wdt K(\eps, H, T), 0\leq t\leq T\right\}\geq 1-\eps \text{ given } |\by|\leq H.
\end{equation}
\end{thm}

Let $(\hat S, \hat P)$ satisfy
\begin{equation}\label{main1}
	\begin{aligned}
		d\hat S(t) &= [D(\xi(t))(S_0-\hat S(t))]dt+\sigma_1\hat  S(t)dW_1(t)\\
				d\hat P(t)&=[D(\xi(t))(P_0-\hat P(t))]dt+\sigma_3\hat P(t)dW_3(t).
	\end{aligned}
\end{equation}

By an easy application of Theorem \ref{thm1} we get the following.
\begin{cor}\label{c:1}
The solution process $(\hat S(t),\hat P(t),\xi(t))$ of \eqref{main1} is a Markov Feller process on $\R^2_+\times\M$ and for any initial value $(s,p,k)\in\R^2_+\times\M$, we have $(\hat S(t),\hat P(t))\in\R^{2,\circ}_+, t\geq0$ with probability 1.
Moreover, we have
\begin{equation}\label{e:mb}
\E_{s,p, k}(1+\hat S(t)+\hat P(t))^q\leq (1+s+p)e^{-c_{1,q}t}+c_{2,q}.
\end{equation}

There is a unique invariant probability measure $\bmu_b$ on $\R^{2}_+\times \M$ and the following two properties hold
\begin{enumerate}
	\item $\bmu_b(\R^{2,\circ}_+\times\M )=1$.
	\item $\sum_{k\in\M}\int_{\R^2_+} (1+s+p)^q\bmu_b(dsdp,k)\leq c_{2,q}, q\leq q_0.$
\end{enumerate}
\end{cor}
\begin{proof}
The moment boundedness from \eqref{e:mb} together with the nondegeneracy of the diffusion implies by \cite{YZ10} the existence of a unique invariant probability measure $\bmu_b$ on $\R^{2}_+\times \M$. We get that $\bmu_b(\R^{2,\circ}_+\times\M )=1$ because any initial value $(s,p,k)\in\R^2_+\times\M$, we have $(\hat S(t),\hat P(t))\in\R^{2,\circ}_+, t\geq0$ with probability 1.
\end{proof}
Using Corollary \ref{c:1} we can define the invasion rate of the microorganisms by
\begin{equation}\label{e:lambda}
\Lambda=\sum_{k\in\M}\int_{\R^2_+} \left(\mu_2(k)s-h_2(k)p-\frac{\sigma_2^2(k)}2\right)\bmu_b(dsdp,k).
\end{equation}
\begin{rem}
Loosely speaking, the invasion rate (also called the Lyapunov exponent) tells us what happens when one introduces the microorganisms at an infinitesimally small density in the ecosystem where the nutrient and flocculants are at stationarity. To see the intuition behind invasion rates and population dynamics we refer the reader to \cite{SBA11, HN16, HNC20}.
\end{rem}
Define 
\begin{equation}\label{e:A}
    \mathbf{A}:=\mathbf{Q} - \diag(D(1),D(2),\cdots, D(m_0)).
\end{equation}
Then $\mathbf{A}$ is nonsingular.
Let
 $$\boldsymbol{\eta}=(\eta(1),\cdots,\eta(m_0))^\top= \mathbf{A}^{-1}\mathbf m_2$$ and
	$$\boldsymbol{\zeta}=(\zeta(1),\cdots,\zeta(m_0))^\top= \mathbf{A}^{-1}\mathbf h_1$$
	where
	$\mathbf m_2=(-\mu_2(1),\cdots, -\mu_2(m_0))^\top$ and $\mathbf h_1=(h_1(1),\cdots, h_1(m_0))^\top$.
    The next lemma shows that we can explicitly compute $\lambda$. 
\begin{lm}\label{lm0}
	We have
	$$
	\Lambda=\sum_{k\in\M}\left[D(k)(\xi(k)+\zeta(k))-\frac{\sigma_2^2(k)}2\right]\pi_k.
	$$
\end{lm} 
Usually invasion-rates cannot be computed explicitly for higher dimensional systems so it is very interesting that it can be done in this setting. 

The next result tells us that if $\lambda$ is negative then the microorganisms go extinct with probability $1$. Furthermore, it also gives the exact exponential rate of extinction.
\begin{thm}\label{thm2}
	If $\Lambda<0$ then for any $\BY(0)=\by\in \R_{+}^{3,\circ}$ we have with probability 1 that
	\begin{equation}\label{thm2-e1}
\lim_{t\to\infty}\frac{\ln X(t)}t=\Lambda
	\end{equation}
    which means that $X$ goes extinct exponentially fast at rate $\Lambda$.
\end{thm}
If $\Lambda>0$ we get stochastic persistence as shown in the next result.
\begin{thm}\label{thm3}
 If $\Lambda>0$  then there exist an invariant measure $\bmu^\circ$ on $\R^{3,\circ}_+\times \M$, $q>1$ such that for any $(\by,k)\in\R^{3,\circ}_{+}\times\M$
	\begin{equation}\label{thm4-e1}
		\lim_{t\to\infty}t^{q-1}\|P_t(\by,k,\cdot)-\bmu^\circ(\cdot)\|_{TV}=0,
	\end{equation}
where $\|\cdot\|_{TV}$ is the total variation metric and $P_t(\by,k,\cdot)=\PP_\by((\BY(t),\xi(t))\in\cdot)$ is the transition probability of the process $(\BY(t),\xi(t))$.
\end{thm}

In Section \ref{s:nonlinear} we explore how nonlinear noise can be treated. 

\section{Proof of the extinction result (Theorem \ref{thm2})}\label{s:ext_1}
We start by first proving Theorem \ref{thm1} and Lemma \ref{lm0}.
	\begin{proof}[Proof of Theorem \ref{thm1}]
	The existence and uniqueness of solutions can be proved similarly to \cite[Theorem 2.1]{zhang2020stationary} or \cite[Theorem 2.1]{nguyen2020general}. The proof for the Markov-Feller property of $(\BY(t))$ can be done by applying \cite[Theorem 5.1]{nguyen2017certain}.
		
Let $\alpha_2=\min\{\frac{\mu_2(k)}{\mu_1(k)},k\in\M\}.$ 
	It is easy to see that we can find $q_0\in(1,2)$ and $\alpha_0>0$ such that the following inequality holds for all $\by\in\R^3_+$
	$$-(q_0-1)(\sigma_1^2(k)(\sigma_1^2(k)s^2+\alpha_2^2\sigma_2^2(k)x^2+ ^2\sigma_3^2(k)p^2)+ D(k)\alpha_0(1+s+\alpha_2x+ p)^2\geq \alpha_0(1+s+\alpha_2x+ p)^2.$$ Define the function
    \begin{equation}\label{e:Uq}
            U^q(\by)=(1+x+\alpha_2y+ z)^q. 
    \end{equation}

    For $0<q\leq q_0$,
	we have
	\begin{equation}\label{LUP}
	\begin{aligned}[\op U^q](\by,k)=&[U^q]_s(\by)\big(D(k)(S_0-s)-\mu_1(k)sx\big)\\
		&+[U^q]_x(\by)\big(\mu_2(k)sx-D(k)x-h_2(k)xp\big)\\
		&+[U^q]_z(\by)\big(D(k)(P_0-p)-h_3(k)xp \big)\\ 
		&+ \frac12\left(\sigma_1^2(k)[U^q]_{ss}s^2+\sigma_2^2(k)[U^q]_{xx}x^2+\sigma_3^3(k)[U^q]_{pp}p^2\right)\\
		\leq &q(D(k)(S_0+ P_0-S(t)-  X(t)- P(t)))(1+s+\alpha_2x+ p)^{q-1}\\&+\frac{q(q-1)}2(1+s+\alpha_2x+ p)^{q-2}(\sigma_1^2s^2+\alpha_2^2\sigma_2^2x^2+ ^2\sigma_3^2p^2)
			\\
			\leq& q\left(D(k)(1+s+\alpha_2x+ p)^{q-1}-q\alpha_0(1+s+\alpha_2 x+p)^{q}\right)\\
			\leq& C_{q}-q\alpha_0 U^q(\by) \,\forall \by\in\R^3_+ \text{ for some constant } C_{q}>0.
		\end{aligned}
	\end{equation}
By Dynkin's formula, we have
	\begin{equation}
	\begin{aligned}
	\E_\by e^{\alpha_0 (t\wedge\bar\tau_n)}U^q(\BY(t\wedge\bar\tau_n))\leq& U^q(\by)+\E_\by\left(\int_0^{t\wedge\bar\tau_n}C_{q_0}e^{q\alpha_0s}ds\right)\\
	\leq& U^q(\by)+ C_{q_0}\int_0^te^{q\alpha_0s}ds\\
\leq& U^q(\by)+ \frac{C_{q_0}}{q\alpha_0}e^{q\alpha_0 t}
	\end{aligned}
\end{equation}
where the stopping time $\tau$ is defined by $\tau_n=\inf\{t\geq 0: U(\by)\geq n\}$.
Dividing both sides by $e^{q\alpha_0 t}$ we obtain \eqref{e1-thm1}.
Similary, with some elementary estimates, we have
$[\op U^2](\by,k)\leq \bar K\, U^2, \forall \by\in\R^3_+,k\in\M.$
Thus, by Dynkin's formula, \begin{equation}\label{e6-thm1}
	\begin{aligned}
	\E_\by e^{-\bar K t}U^2(\BY(t\wedge\bar\tau_n))\leq 	\E_\by e^{-\bar K (t\wedge\bar\tau_n)}U^2(\BY(t\wedge\bar\tau_n))\leq& U^2(\by).
	\end{aligned}
\end{equation}
Letting $n\to\infty$, we can derive \eqref{e2-thm1} from Lebesgue’s dominated convergence theorem. One can then obtrain \eqref{e3-thm1} from \eqref{e6-thm1}.
\end{proof}

Let $c_0$ be sufficiently large such that
\begin{equation}\label{c_0}
	\begin{aligned}
-2&d(k)[(s-\wdt s)^2+(p-\wdt p)^2]+ \sigma_1^2(k)(s-\wdt s)^2+\sigma_3^2(k)(p-\wdt p)^2+ 2(s-\wdt s)\mu_1(k)\wdt s\wdt x+2(p-\wdt p)h_2(k)\wdt p\wdt x\\
& \leq c_0\left(s-\wdt s)^2+(p-\wdt p)^2 + (\wdt s+ \wdt p)^2\wdt x^2\right) \text{ for all } s,\wdt s, p,\wdt p, \wdt x>0.
\end{aligned}
\end{equation}
Since in this subsection we look at extinction, we remind the reader that $\Lambda<0$.

Let $\gamma_0:=-\frac{\Lambda}3>0$ and pick $\lambda> \gamma_0 +\sigma_1^2\vee\sigma_2^2 +\frac{c_0}2$.
We will consider 
the following coupling process.
\begin{equation}\label{coup1}
	\begin{cases}
		d\BS(t)=& [D(\xi(t))(S_0-\bar S(t))]dt+\sigma_1(\xi(t)) \BS(t)dW_1(t)\\
		d\BP(t)=& [D(\xi(t))(P_0-\bar P(t))]dt+\sigma_2(\xi(t)) \BP(t)dW_3(t)\\
		d\wdt S(t)=&[D(\xi(t)(S_0-\wdt S(t))-\mu_1(\xi(t))\wdt S(t)\wdt X(t)]dt+\sigma_1(\xi(t))\wdt  S(t)dW_1(t)\\
		&+\lambda\1_{\{t<\wdt\tau_\delta\}}(\BS(t)-\wdt S(t))dt\\
		d\wdt P(t)=& [D(\xi(t)(P_0-\wdt P(t))-h_2\wdt X(t)\wdt P(t)]dt+\sigma_3\wdt  P(t)dW_3(t)\\
		&+\lambda\1_{\{t<\wdt\tau_\delta\}}(\BP(t)-\wdt P(t))dt\\
		d\wdt X(t)=&\wdt X(t) \left[\mu_2(\xi(t))\wdt S(t)-h_2(\xi(t))\wdt P(t)\right] dt+\sigma_2(\xi(t))\wdt  X(t)dW_2(t)\\
	\end{cases}
\end{equation}

This coupling will help us prove the extinction results by ensuring that we can control certain processes and keep them close on an infinite time horizon with high probability.

\begin{proof}[Proof of Lemma \ref{lm0}]
		Since $q_{ij}>0$ for $i\ne j$ and $q_{ii}=-\sum_{j\ne i} q_{ij}$, it is easy to see that 
		$\mathbf A$ is a strictly diagonally dominant matrix. As a result, $\mathbf A$ is nonsingular and $\mathbf A^{-1}$ exists.
	By It\^o's formula, we have
	\begin{equation}\label{e1-lm0}
		\begin{aligned}
	\E_{\bmu_b} \left[\eta(\xi(T))\BS(T)+\zeta(\xi(T))\BP(t)\right]=&
	\E_{\bmu_b} \left[\eta(\xi(0))\BS(0)+\zeta(\xi(0))\BP(0)\right]\\
	&+\E_{\bmu_b}\int_0^T D(\xi(t))[\eta(\xi(t))+\zeta(\xi(t))]dt\\
	&+\E_{\bmu_b}\int_0^T\left(-D(\xi(t))\eta(\xi(t))+\sum_{j\in\M}q_{\xi(t),j}\eta(j)\right)\BS(t)dt\\
		&+\E_{\bmu_b}\int_0^T\left(-D(\xi(t))\zeta(\xi(t))+\sum_{j\in\M}q_{\xi(t),j}\zeta(j)\right)\BP(t)dt\\
=&\E_{\bmu_b} \left[\eta(\xi(0))\BS(0)+\zeta(\xi(0))\BP(0)\right]\\
&+\E_{\bmu_b}\int_0^T D(\xi(t))[\eta(\xi(t))+\zeta(\xi(t))]dt\\
&+\E_{\bmu_b}\int_0^T \left[-\mu_2(\xi(t))\BS(t)+h_1(\xi(t))\BP(t)\right]dt		
	\end{aligned}
\end{equation}
where the last equality comes from the fact that $-D(k)\eta(k)+\sum_{j\in\M}q_{kj}\eta(k)=-\mu_2(k)$ and $-D(k)\zeta(k)+\sum_{j\in\M}q_{kj}\zeta(k)=h_1(k)$.
In the equation above the subscript of $\E_{\bmu_b}$ indicates that we condition on the initial distribution $\bmu_b$. This implies that $(\BS(t),\BP(t),\xi(t))$ has distribution $\bmu_b$ for any $t\geq 0$ since $\bmu_b$ is an invariant probability measure.
As a result, 
\begin{equation}\label{e2-lm0}
\E_{\bmu_b} \left[\eta(\xi(T))\BS(T)+\zeta(\xi(T))\BP(t)\right]=\E_{\bmu_b} \left[\eta(\xi(0))\BS(0)+\zeta(\xi(0))\BP(0)\right],
\end{equation}
\begin{equation}\label{e3-lm0}
	\E_{\bmu_b}\int_0^T \left[-\mu_2(\xi(t))\BS(t)+h_1(\xi(t))\BP(t)\right]dt	
	=T\sum_{k\in\M}\int_{\R^2_+}\left(\eta(k)s+\zeta(k)p\right)\bmu_b(dsdp)
\end{equation}
and
\begin{equation}\label{e4-lm0}
	\begin{aligned}
\E_{\bmu_b}\int_0^T D(\xi(t))[\eta(\xi(t))+\zeta(\xi(t))]dt
=&T\sum_{k\in\M}\int_{\R^2_+}\left(D(k)(\xi(k)+\zeta(k))\right)\bmu_b(dsdp)\\
=&T\sum_{k\in\M}(D(k)(\xi(k)+\zeta(k))\pi_k
\end{aligned}
\end{equation}
where the last equality is due to the fact that the process $\xi(t)$ has a unique invariant probability measure $\pi$.
Plugging \eqref{e2-lm0}, \eqref{e3-lm0} and \eqref{e4-lm0} into \eqref{e1-lm0}, we have
$$
\sum_{k\in\M}\int_{\R^2_+}\left(\eta(k)s+\zeta(k)p\right)\bmu_b(dsdp)=\sum_{k\in\M}(D(k)(\xi(k)+\zeta(k))\pi_k.
$$
Thus,
$$
\Lambda=\sum_{k\in\M}\int_{\R^2_+}\left(\eta(k)s+\zeta(k)p\right)\bmu_b(dsdp)-\sum_{k\in\M}\frac{\sigma_2^2(k)}2=\sum_{k\in\M}\left[D(k)(\xi(k)+\zeta(k))-\frac{\sigma_2^2(k)}2\right]\pi_k.
$$
\end{proof}
\begin{lm}\label{lm1}
	For any $\delta>0$, $s,p\geq0$, $\eps\in(0,1)$, there exists $M_0=M_0(\eps, \delta, s,p)$ such that
	$$\PP_{s,p,k}\left\{\BS(t)+\BP(t)+\BS^{-1}(t)+\BP^{-1}(t)\leq M_0e^{\delta t}, t\geq 0\right\}\geq 1-\eps.$$
\end{lm}
\begin{proof}
	Consider the Markov process $(\BS(t),\xi(t))$ where $\BS(t)$ satisfying	$$d\BS(t)= [D(\xi(t))(S_0-S(t))]dt+\sigma_1(\xi(t)) \BS(t)dW_1(t),\BS(0)=s>0.$$
Let $\Phi(s)=s+s^{-1}$.
By It\^o's formula, and a standard algebraic inequality, we have
$$
\begin{aligned}
	d\Phi(\BS(t))=&\left[D(\xi(t))(S_0-\BS(t))-D(\xi(t))\frac{(S_0-\BS(t))}{\BS^2(t)}+\sigma_1^2(\xi(t))\frac{1}{\BS(t)}\right]dt\\
	&+\sigma_1(\xi(t))\BS(t)dW_1(t)-\sigma_1(\xi(t))\frac{1}{\BS(t)}dW_1(t)\\
	\leq& [C_\Phi - D_m\Phi(\BS(t))]dt+\sigma_1(\xi(t))\BS(t)dW_1(t)-\sigma_1(\xi(t))\frac{1}{\BS(t)}dW_1(t)
\end{aligned}
$$	
where $D_m=\min\{D(k): k\in\M\}$ and $C_\phi$ is a sufficiently large constant.
Applying Dynkin's formula, we deduce that
\begin{equation}\label{eD_m}
	\E_{s,k}e^{D_m(\bar\tau_{L}\wedge t)}\Phi(\BS(\bar\tau_{L}\wedge t))\leq \Phi(s) + \E_{s,k} \int_0^{\bar\tau_{L}\wedge t} C_\Phi e^{D_m s}ds\leq \Phi(s)+\frac{C_\Phi e^{D_m t}}{D_m},
	\end{equation}
where $\bar\tau_{L}:=\inf\{t\geq0: \Phi(\BS(t))\geq L\}.$
Letting $L\to\infty$ and then diving both sides of \eqref{eD_m} by $e^{D_mt}$ we have
\begin{equation}\label{ePhi1}
	\E_{s,k}\Phi(\BS(\bar\tau_{L}\wedge t))\leq \Phi(s) + \frac{C_\Phi }{D_m}, t\geq 0
\end{equation}

On the other hand, for any $ L >0$, applying Markov's inequality  we have from \eqref{eD_m} that
$$
\PP\left\{\sup_{t\in[0,1]}\bar \Phi(S(t))\geq  L \right\}\leq \frac1{ L }\E_{s,k} \Phi(\BS(\bar\tau_{ L }\wedge 1))\leq  \frac1L\left(\Phi(s)+\frac{C_\Phi e^{D_m}}{D_m}\right)
$$
By the Markov property of $(\BX,\xi(t))$ and by applying the inequality above for $L=M_{01} e^{\delta_n}$ we have
$$
\PP\left\{\sup_{t\in[n,n+1]}\bar V(X(t),Y(t))> M_{01} e^{\delta n}\right\}\leq \frac{1}{M_{01} e^{\delta n}}\left(\E_{s,k} \Phi(\BS(n))+\frac{C_\Phi e^{D_m}}{D_m}\right)\leq \frac{1}{M_{01} e^{\delta n}}\left(\Phi(s)+\frac{C_\Phi }{D_m}+\frac{C_\Phi e^{D_m}}{D_m}\right)
$$
where the last inequality is due to \eqref{ePhi1}.
As a result, 
\begin{equation}\label{e5-lm1}
	\PP\left\{\sup_{t\in[n,n+1]}\bar V(X(t),Y(t))\leq M_{01} e^{\delta\theta n}, \text{ for all } n\in\N\right\}>1-\left(\Phi(s)+\frac{C_\Phi }{D_m}+\frac{C_\Phi e^{D_m}}{D_m}\right)\sum_{n=1}^\infty\frac{1}{M_{01} e^{\delta n}}
\end{equation}
For $\eps>0$, pick $ M_{01}=M_{01}(s,\eps,\delta)$ such that $\left(\Phi(s)+\frac{C_\Phi }{D_m}+\frac{C_\Phi e^{D_m}}{D_m}\right)\sum_{n=1}^\infty\frac{1}{M_{01} e^{\delta n}}\leq\frac\eps2$, we have
	$$\PP_{s,k}\left\{\BS(t)+\BS^{-1}(t)\leq M_{01}e^{\delta t}, t\geq 0\right\}\geq 1-\frac\eps2 \text{ for all } \eps\in(0,1).$$
	Similarly, we can find $M_{02}$ such that
		$$\PP_{p,k}\left\{\BP(t)+\BP^{-1}(t)\leq M_{02}e^{\delta t}, t\geq 0\right\}\geq 1-\frac\eps2 \text{ for all } \eps\in(0,1).$$
Combining two displays above we can complete the proof.
\end{proof}

\begin{prop}\label{prop2.1}
	We have that
	\begin{equation}\label{e1-prop2.1}
		\E \sup_{0\leq t\leq\wdt\tau_\delta}e^{\gamma_0 t}[(\BS(t)-\wdt S(t))^2+(\BP(t)-\wdt P(t))^2] \leq \wdt C((\bar x-\wdt x)^2+(\bar y-\wdt y)^2+\delta^2)
	\end{equation}
where 
$$\wdt\tau_\delta=\inf\left\{t\geq0: (\wdt S(t)+ \wdt P(t))\wdt X(t)\geq\delta e^{-\gamma_0 t}\right\}.$$
Moreover, there are $\wdt M_{\eps, x,y},\wdt m_{\eps,x,y}>0$ such as
	\begin{equation}\label{e2-prop2.1}
\PP_{x,y,\wdt\by}\left\{\int_0^{\wdt\tau_\delta}|v_1(t)|^2+|v_2(t)|^2)dt\geq \wdt M_{\eps, x,y}((s-\wdt s)^2+(p-\wdt p)^2+\delta^2) \right\}\leq \eps
	\end{equation}
as long as $(s-\wdt s)^2+(p-\wdt p)^2+\delta^2\leq \wdt m_{\eps,x,y}$
where
$$v_1(t)=\frac{\lambda(\BS(t)-\wdt S(t))}{\sigma_1(\xi(t))\wdt S(t)}\text{ and }
v_3(t)=\frac{\lambda(\BP(t)-\wdt P(t))}{\sigma_3(\xi(t))\wdt P(t)}.$$
\end{prop}
\begin{proof}

Using \eqref{c_0} we get
\begin{equation}\label{e01-prop2.1}
	\begin{aligned}
	d[(\BS(t)-\wdt S(t))^2&+(\BP(t)-\wdt P(t))^2]\\
	=\Big[&-2D(\xi(t))[(\BS(t)-\wdt S(t))^2+(\BP(t)-\wdt P(t))^2\\
&	+\sigma_1^2(\xi(t))(\BS(t)-\wdt S(t))^2+\sigma_3^2(\xi(t))(\BP(t)-\wdt P(t))^2\\
&+	2(\BS(t)-\wdt S(t))\mu_1(\xi(t))\wdt S(t)\wdt X(t)+2(\BP(t)-\wdt P(t))h_2\wdt P(t)\wdt X(t)\\	
&-2\gamma[(\BS(t)-\wdt S(t))^2+(\BP(t)-\wdt P(t))^2\Big]dt\\
&+2\sigma_1(\xi(t))(\BS(t)-\wdt S(t))^2dW_1(t)+2\sigma_3(\xi(t))(\BP(t)-\wdt P(t))^2dW_3(t)\\
\leq\quad & (-2\gamma+c_0)[(\BS(t)-\wdt S(t))^2+(\BP(t)-\wdt P(t))^2]dt\\
&+ c_0(\wdt S(t)+\wdt P(t))^2 \wdt X^2(t)dt\\
&+2\sigma_1(\xi(t))(\BS(t)-\wdt S(t))^2dW_1(t)+2\sigma_3(\xi(t))(\BP(t)-\wdt P(t))^2dW_3(t)\\
	\end{aligned}
\end{equation}
By It\^o's formula and Cauchy's inequality, we have
\begin{equation}
	\begin{aligned}
d&e^{2\gamma_0 t}[(\BS(t)-\wdt S(t))^2+(\BP(t)-\wdt P(t))^2]^2\\&\leq -(4\gamma-4\gamma_0+4c_0+4\check\sigma^2)e^{4\gamma_0 t}[(\BS(t)-\wdt S(t))^2+(\BP(t)-\wdt P(t))^2]+ c_0e^{4\gamma_0 t}(\wdt S(t)+\wdt P(t))^4|\wdt Z(t)|^4dt\\
&\qquad+ 4e^{4\gamma_0 t}[(\BS(t)-\wdt S(t))^2+(\BP(t)-\wdt P(t))^2]^2\left(\sigma_1(\xi)(\BS(t)-\wdt S(t))^2dW_1(t)+\sigma_2(\xi)(\BP(t)-\wdt P(t))^2dW_2(t)\right).
	\end{aligned}
\end{equation}
where $\check\sigma=\max\{|\sigma_1(k)|,|\sigma_2(k)|,k\in\M\}$.
Let $\hat\tau_n=\wdt\tau_\delta\wedge \inf\{n\geq 0: (\BS(t)-\wdt S(t))^2+(\BP(t)-\wdt P(t))^2\geq n\}$, we have from Dynkin's formula that
\begin{align*}
\E_{s,p,\wdt \by,k}&e^{4\gamma_0 t}[(\BS(t)-\wdt S(t))^2+(\BP(t\wedge\hat\tau_n)-\wdt P(\wedge\hat\tau_n))^2]^2
+\E_{s,p,\wdt \by,k} \int_0^{t\wedge\wdt\tau_\delta} e^{4\gamma_0 s}[(\BS(s)-\wdt S(s))^2+(\BP(s)-\wdt P(s))^2]^2\\
\leq&
\E \int_0^{t\wedge\hat\tau_n}c_0e^{4\gamma_0 s}(\wdt S(t)+\wdt P(t))^4|\wdt Z(s)|^4ds
+ ((s-\wdt s)^2+(p-\wdt p)^2)^2,
\end{align*}
Letting $n\to\infty$ we have from Lebesque's dominated theorem that
\begin{align*}
(4\gamma-4\gamma_0+2c_0+4\check\sigma^2)&\E \int_0^{t\wedge\wdt\tau_\delta} e^{4\gamma_0 s}[(\BS(s)-\wdt S(s))^2+(\BP(s)-\wdt P(s))^2]^2\\
\leq&
\E \int_0^{t\wedge\wdt\tau_\delta}c_0e^{4\gamma_0 s}(\wdt S(t)+\wdt P(t))^4|\wdt X(s)|^4ds
+ ((\bar x-\wdt x)^2+(\bar y-\wdt y)^2)^2.
\end{align*}
Consequently, we have from the definition of $\wdt\tau_\delta$ that
\begin{equation}\label{e3-prop2.1}
\E \int_0^{t\wedge\wdt\tau_\delta} e^{4\gamma_0 s}[(\BS(s)-\wdt S(s))^2+(\BP(s)-\wdt P(s))^2]^2\leq C((\bar x-\wdt x)^4+(\bar y-\wdt y)^4+\delta^4)\end{equation}
Here and thereafter, $C$ is a generic constant independent of $|\bar x-\wdt x|, |\bar y-\wdt y|$ and $\delta$ and $C$ can be different in different lines.
Now, applying It\^o's fomula to \eqref{e01-prop2.1}, we have
\begin{equation}
	\begin{aligned}
		d&e^{2\gamma_0 t}[(\BS(t)-\wdt S(t))^2+(\BP(t)-\wdt P(t))^2]\\&\leq -(\lambda-2\gamma_0-c_0)e^{2\gamma_0 t}[(\BS(t)-\wdt S(t))^2+(\BP(t)-\wdt P(t))^2]+ c_0e^{2\gamma_0 t}(\wdt S(t)+\wdt P(t))^2 \wdt X^2(t)dt\\
		&\qquad+ e^{2\gamma_0 t}\left(\sigma_1(\xi(t))(\BS(t)-\wdt S(t))^2dW_1(t)+\sigma_3(\xi(t))(\BP(t)-\wdt P(t))^2dW_2(t)\right).
	\end{aligned}
\end{equation}
Then taking expectation yields
\begin{equation}\label{e4-prop2.1}
\begin{aligned}
\E \sup_{t\leq T\wedge\wdt\tau_\delta}& e^{2\gamma_0 t}[(\BS(t)-\wdt S(t))^2+(\BP(t)-\wdt P(t))^2]\\
\leq &\E \int_0^{T\wedge\wdt\tau_\delta} c_0e^{2\gamma_0 s}(\wdt S(s)+\wdt P(s))^2 \wdt X^2(t)ds\\
&+\E \sup_{t\leq T\wedge\wdt\tau_\delta}\int_0^t\left\{e^{2\gamma_0 s}\left(\sigma_1(\xi(s))(\BS(s)-\wdt S(s))^2dW_1(s)+\sigma_2(\xi(s))(\BP(s)-\wdt P(s))^2dW_2(s)\right)\right\}
\end{aligned}
\end{equation}
In view of the Burkholder-Davis-Gundy inequality, we have
\begin{equation}\label{e5-prop2.1}
	\begin{aligned}
		\E \sup_{t\leq T\wedge\wdt\tau_\delta}&\int_0^t\left\{e^{2\gamma_0 s}\left(\sigma_1(\BS(s)-\wdt S(s))^2dW_1(s)+\sigma_2(\BP(s)-\wdt P(s))^2dW_2(s)\right)\right\}\\
		\leq& C\left[\E \int_0^{t\wedge\wdt\tau_\delta} e^{4\gamma_0 s}[(\BS(s)-\wdt S(s))^2+(\BP(s)-\wdt P(s))^2]^2\right]^{\frac12}\\
		\leq& C((\bar x-\wdt x)^2+(\bar y-\wdt y)^2+\delta^2) \text{ (due to \eqref{e3-prop2.1})}
	\end{aligned}
\end{equation}
and clearly from the definition of $\wdt\tau_\delta$,
\begin{equation}\label{e6-prop2.1}
\E \int_0^{T\wedge\wdt\tau_\delta} c_0e^{2\gamma_0 s}(\wdt S(s)+\wdt P(s))^2\wdt X^2(s)ds\leq C\delta^2.
\end{equation}

Applying \eqref{e5-prop2.1} and \eqref{e6-prop2.1} to \eqref{e4-prop2.1}, we can easily obtain \eqref{e1-prop2.1}.

Next, let
$$v_1(t)=\frac{\lambda(\BS(t)-\wdt S(t))}{\sigma_1(\xi(t))\wdt S(t)}$$
and
$$v_3(t)=\frac{\lambda(\BP(t)-\wdt P(t))}{\sigma_3(\xi(t))\wdt P(t)}.$$ In view of Lemma \ref{lm1},
\begin{equation}\label{e9-prop2.1}
\PP_{x,y}\left(\wdt\Omega_3:=\left\{[\BS(t)^{-1}+\BP^{-1}(t)]^2\leq M_{\eps,s, p, 1}e^{\frac{\gamma_0 t}2}, \text{ for all } t\geq 0\right\}\right)\geq 1-\frac{\eps}2.
\end{equation}
By virtue of \eqref{e1-prop2.1}, there is $\wdt C_0$ independent of $(\bar x-\wdt x)^2+(\bar y-\wdt y)^2+\delta^2$ such that
$\PP_{x,y,\wdt\by}(\wdt\Omega_4)\geq 1-\frac{\eps}2$ where
\begin{equation}\label{e10-prop2.1}
\wdt\Omega_4:=\left\{e^{2\gamma_0 t}[(\BS(t)-\wdt S(t))^2+(\BP(t)-\wdt P(t))^2] \leq \frac{\wdt C_0((\bar x-\wdt x)^2+(\bar y-\wdt y)^2+\delta^2)}\eps \text{ for all } 0\leq t\leq \wdt\tau_\delta \right\}.
\end{equation}

For $t\leq\tau_\delta$, if $\BS(t)\geq M_{\eps,s, p, 1}^{-1} e^{-\frac{\gamma_0 t}2}$ and
$(\BS(t)-\wdt S(t))\leq \frac12M_{\eps,s, p, 1}^{-1} e^{-\frac{\gamma_0 t}2}$, we have
\begin{equation}\label{e6a-prop2.1}
	\frac1{\wdt S(t)}\leq \frac1{\BS(t)+(\BS(t)-\wdt S(t))}\leq \frac{1}{M_{\eps,s, p, 1}^{-1} e^{-\frac{\gamma_0 t}2}+(\BS(t)-\wdt S(t))}\leq 2 M_{\eps,s, p, 1} e^{\frac{\gamma_0 t}2}
\end{equation}

Analogously,
\begin{equation}\label{e6b-prop2.1}
	\frac1{\wdt P(t)}\leq 2 M_{\eps,s, p, 1} e^{\frac{\gamma_0 t}2}
	\text{ when }\BS(t)\geq M_{\eps,s, p, 1}^{-1} e^{-\frac{\gamma_0 t}2} \text{ and
	} (\BS(t)-\wdt S(t))\leq \frac12M_{\eps,s, p, 1}^{-1} e^{-\frac{\gamma_0 t}2}.
\end{equation}

If $(\bar x-\wdt x)^2+(\bar y-\wdt y)^2+\delta^2)\leq \frac{\eps}{4\wdt C_0M^2_{\eps,x,y, 1}}$
and $\omega\in\wdt\Omega_3$ then
$$(\BS(t)-\wdt S(t))\vee (\BP(t)-\wdt P(t))\leq  \left(\frac{\wdt C_0((\bar x-\wdt x)^2+(\bar y-\wdt y)^2+\delta^2)}\eps\right)^{-\frac12}  e^{-\frac{\gamma_0 t}2}\leq \frac12M_{\eps,s, p, 1}^{-1} e^{-\frac{\gamma_0 t}2}.$$

This together with \eqref{e6a-prop2.1} and \eqref{e6b-prop2.1} implies that
\begin{equation}\label{e6c-prop2.1}
	\frac1{\wdt S(t)}\vee	\frac1{\wdt P(t)}\leq 2M_{\eps,s, p, 1}e^{\frac{\gamma_0 t}2}
\end{equation}
for $\omega\in\wdt\Omega_3\bigcap\wdt\Omega_4$ and $(\bar x-\wdt x)^2+(\bar y-\wdt y)^2+\delta^2)\leq \frac{\eps}{2\wdt C_0M_{\eps,x,y, 1}}$.

Note that,
$$|v_1(t)|^2+|v_2(t)|^2\leq \frac{2\lambda^2}{\hat\sigma^2}\left(\wdt S^{-2}(t)\wedge\wdt P^{-2}(t)\right)[(\BS(t)-\wdt S(t))^2+(\BP(t)-\wdt P(t))^2]$$
where $\hat\sigma=\min\{|\sigma_1(k)|,|\sigma_3(k)|,k\in\M\}$.
Putting \eqref{e9-prop2.1} and \eqref{e10-prop2.1} and \eqref{e6c-prop2.1} together, we have
$$
\PP\left\{|v_1(t)|^2+|v_2(t)|^2\leq \frac{8\lambda^2 M^2_{\eps,s,p,1}}{\hat\sigma^2}\frac{\wdt C_0((\bar x-\wdt x)^2+(\bar y-\wdt y)^2+\delta^2)}\eps e^{-\gamma_0 t/2}\text{ for all } 0\leq t\leq \wdt\tau_\delta\right\}\geq 1-\eps,
$$
when $(\bar x-\wdt x)^2+(\bar y-\wdt y)^2+\delta^2)\leq \frac{\eps}{2\wdt C_0M_{\eps,s,p, 1}}=:\wdt m_{\eps,s,p}$. Now we can see that \eqref{e2-prop2.1} also follows easily.
\end{proof}
\begin{prop}\label{lm2}
	For any $(s,p)\in\R^{2,\circ}_+$ and $\eps\in(0,1)$, there exists $\varsigma=\varsigma(s,p,\eps)>0$ such that
	$$
	\PP_{\wdt\by}\left\{\lim_{t\to\infty}\frac{\ln \tilde X(t)}t=\Lambda<0\right\}>1-\eps
	$$
	if $(\wdt s-s)^2+(\wdt p-p)^2+\wdt x^2\leq\varsigma^2$.
\end{prop}
\begin{proof}
	First, we choose $\delta=\delta(\eps,s,p)\in\left(0,\frac12\right)$ such that
	\begin{equation}\label{wdtM}
	2\wdt M_{\eps, s,p}\delta^2\leq\eps \text{
	and }2\eps^2\wdt M_{\eps, s,p}2\delta\leq\eps
\text{ and } \frac{\wdt C_0 2\delta^2}\eps\leq 1
\end{equation}
	where $\wdt M_{\eps,s,p}$ is determined as in \eqref{e2-prop2.1}.
We will choose $\varsigma\in (0,\delta)$ later. Now, we assume that
$(\wdt s-s)^2+(\wdt p-p)^2+\wdt x^2\leq\varsigma^2$ which leads to
$(\wdt s-s)^2+(\wdt p-p)^2+\delta^2\leq2\delta^2$.

Define
$$
\Omega_1:=\left\{e^{2\gamma_0 t}[(\BS(t)-\wdt S(t))^2+(\BP(t)-\wdt P(t))^2] \leq \frac{\wdt C_0 ((s-\wdt s)^2+(p-\wdt p)^2+\delta^2)}\eps\leq 1\right\}
$$
Because of the egodicity of $(\BS(t),\BP(t),\xi(t))$, we have that
$$
\PP_{s,p,k}\left\{\frac1t\int_0^t \left(\mu_2(\xi(t))\BS(t)+h_1(\xi(t))\BP(t)-D(\xi(t))-\frac{\sigma_2^2(\xi(t))}2\right) dt=\lambda<0\right\}=1.$$
We also get that if we define
$$
\Omega_2:=\left\{\frac1t\int_0^t \left(\mu_2(\xi(s))\BS(s)+h_1(\xi(s))\BP(s)-D(\xi(s))-\frac{\sigma_2^2(\xi(s))}2\right)ds\leq \Lambda+\gamma_0, t\geq T\right\}
$$
there is $T>0$ such that $\PP_{s,p}(\Omega_2)>1-\eps$.

Letting 
$$
\Omega_3:=\left\{\int_0^t \left(\mu_2(\xi(s))\BS(s)+h_1(\xi(s))\BP(s)-D(\xi(s))-\frac{\sigma_2^2(\xi(s))}2\right)ds\leq \wdt D_{s,p,\eps, T}, t\leq T\right\}
$$
we see that in view of \eqref{e3-thm1}  we can find $\wdt D_{s,p,\eps, T}>0$ such that $\PP_{s,p}(\Omega_3)\geq 1-\eps.$

Moreover, if we set
$$
\Omega_4:=\left\{\int_0^t\sigma_2(\xi(s))dW_2(s)\leq \frac{\check\sigma^2}{2\gamma_0}|\ln\eps|+\gamma_0 t, \forall\, t\geq0\right\}
$$
using the exponential martingale inequality, see e.g. \cite{mao}, we get $\PP(\Omega_4)\geq 1-\eps$.

	By Lemma \ref{lm1}, there is $M_2(\eps,s,p)$ such that $\Omega_5\geq 1-\eps$, where
	$$
\Omega_5=\{\BS(t)+\BP(t)+\BS^{-1}(t)+\BP^{-1}(t)\leq M_2(\eps,s,p)e^{\gamma_0 t}\}.
$$
Let $K=\max\{h_1(k), \mu_2(k),k\in\M\}$,
for $0\leq t\leq T\wedge\wdt\tau_\delta$, $\omega\in\bigcap_{i=1}^5\Omega_i$ we have that
\begin{equation}\label{e11-prop2.1}
	\begin{aligned}
		\ln\wdt X(t)=&\ln\wdt x+\int_0^t \left(\mu_2(\xi(s))\wdt S(s)+h_1(\xi(s))\wdt P(s)-D(\xi(s))-\frac{\sigma_2^2(\xi(s))}2\right)ds
		+\int_0^t\sigma_2(\xi(s))W_2(s)\\
	\leq& \ln\wdt x+\int_0^t \left(\mu_2(\xi(t))\BS(t)+h_1(\xi(t))\BP(t)-D(\xi(t))-\frac{\sigma_2^2(\xi(t))}2\right) dt+\int_0^t\sigma_2(\xi(s))W_2(s)\\
	&+K\int_0^t |(\wdt Z(t))^2+(\BP(s)-\wdt P(s))^2|^{\frac12}ds\\
	\leq& \ln\wdt x +\frac{\check\sigma^2}{2\gamma_0}|\ln\eps| +\wdt D_{s,p,\eps, T}+\gamma_0 t+K\int_0^t e^{-\gamma_0s} ds\\
	\leq&\ln\wdt x +\frac{\check\sigma^2}{2\gamma_0}|\ln\eps| + \gamma_0t+ \wdt D_{s,p,\eps, T}+\frac{K}{\gamma_0},
	\end{aligned}
\end{equation}
and that
\begin{equation}\label{e11b-prop2.1}
	\begin{aligned}
\ln (\wdt S(t)+\wdt P(t))\leq& \ln (\BS(t)+\BP(t)+1)
\leq & \ln M_{\eps,x,y,2}+\gamma_0t .
\end{aligned}
\end{equation}
If $\ln\wdt x< \ln\varsigma:=\ln\delta - \left( \ln M_{\eps,x,y,2}+\frac{\check\sigma^2}{2\gamma_0}|\ln\eps| + \gamma_0T+ \wdt D_{s,p,\eps, T}+\frac{K}{\gamma_0} \right)$
then it is easily seen that
$\wdt\tau_\delta\geq T$ for any $\omega\in\bigcap_{i=1}^5\Omega_i$ because $$\ln\wdt X(t)+\ln (\wdt S(t)+\wdt P(t))\leq \ln\delta$$ for any $t\leq T\wedge\wdt\tau_\delta$ for any $\omega\in\bigcap_{i=1}^5\Omega_i$. For $T\leq t\leq \wdt\tau_\delta$ we have from \eqref{e11-prop2.1} and \eqref{e11b-prop2.1} that
$$
\ln \wdt X(t)+\ln (\wdt S(t)+\wdt P(t))\leq \ln\wdt x+ (\Lambda+3\gamma_0)t+\frac{\check\sigma^2}{2\gamma_0}|\ln\eps| + \wdt D_{s,p,\eps, T}+\frac{K}{\gamma_0}+\ln M_{\eps,x,y,2}  <\ln\delta.$$
In particular, we have
$$
\ln \wdt X(t)+\ln (\wdt S(t)+\wdt P(t))\leq \ln\wdt x+ (\Lambda+2\gamma_0)t+\frac{\check\sigma^2}{2\gamma_0}|\ln\eps| + \wdt D_{s,p,\eps, T}+\frac{K}{\gamma_0}$$
for $\omega\in\bigcap_{i=1}^5\Omega_i$
which leads to
\begin{equation}\label{lnwdtX}
	\limsup_{t\to\infty} \frac{\ln\wdt X(t)}t\leq \Lambda-2\gamma_0<0 \text{ for }\omega\in\bigcap_{i=1}^5\Omega_i.
\end{equation}
An application of the Cameron-Martin-Girsanov theorem implies that
 under the measure $\Q_{s,p,\wdt\by}$ defined by
$$
\dfrac{d\Q_{s,p,\wdt\by}}{d\PP_{s,p,\wdt\by}}=\exp\left\{-\int_0^{\wdt\tau_\delta}[v_1(s)dW_1(s)+v_2(s)dW_2(s)]-\int_0^{\wdt\tau_\delta}[v_1^2(s)+v_2^2(s)]ds\right\},
$$
$\left(W_1(t)+\int_0^{t\wedge\wdt\tau_\delta}v_1(s)ds,W_2(t)+\int_0^{t\wedge\wdt\tau_\delta}v_2(s)ds\right)$ is a standard two-dimensional brownian motion under $\Q$.
As a result,
$(\wdt S(t), \wdt X(t),\wdt P(t))$ is the solution to \eqref{main} with initial condition $(\wdt\by,k)$ under $\Q$.

Let
$$
\Omega_6:=\left\{\int_0^{\wdt\tau_\delta}|v_1(t)|^2+|v_2(t)|^2)dt\geq \wdt M_{\eps,s,p}\right\}$$
and
$$
\Omega_7:=\left\{\int_0^t (v_1(s)dW_1(s)+v_2(s)dW_2(s))\leq\frac{\eps^2}{2\delta} \int_0^{t}|v_1(s)|^2+|v_2(s)|^2)ds+\eps\right\}.
$$
In view of the exponential martingale inequality, we have
$$
\PP_{s,p,\wdt\by}(\Omega_7)\geq 1-e^{\eps^3/\delta}\geq 1-\eps
$$
if $\delta\leq \eps^3/(-\ln\eps)$.
For $\omega\in\Omega_6\bigcap\Omega_7$, we get that
\begin{equation}
	\begin{aligned}
		\dfrac{d\Q_{s,p,\wdt\by,k}}{d\PP_{s,p,\wdt\by,k}}=&\exp\left\{-\int_0^{\wdt\tau_\delta}[v_1(s)dW_1(s)+v_2(s)dW_2(s)]-\int_0^{\wdt\tau_\delta}[v_1^2(s)+v_2^2(s)]ds\right\}\\
		\geq & \exp\left\{-\frac{\eps^2}{2\delta} \int_0^{t}|v_1(s)|^2+|v_2(s)|^2)ds-\eps-\int_0^{\wdt\tau_\delta}[v_1^2(s)+v_2^2(s)]ds\right\}\\
		\geq & e^{-\frac{\eps^2 \wdt M_{\eps,x,y}2\delta^2}{2\delta}-\eps -\wdt M_{\eps,x,y}2\delta^2}\geq e^{-3\eps}\geq 1-4\eps \text{ (due to \eqref{wdtM})}.
	\end{aligned}
\end{equation}
As a result,
$$\Q_{s,p,\wdt\by,k}\left(\bigcap_{i=1}^7\Omega_i\right)=\int_{\bigcap_{i=1}^6\Omega_i} \dfrac{d\Q_{s,p,\wdt\by,k}}{d\PP_{s,p,\wdt\by,k}} d\PP_{s,p,\wdt\by,k}\geq (1-4\eps)\PP_{s,p,\wdt\by,k}\left(\bigcap_{i=1}^6\Omega_i\right)\geq (1-4\eps)(1-6\eps)=1-10\eps.
$$

Note that for $\omega\in\bigcap_{i=1}^7\Omega_i$, we have $\wdt\tau_\delta=\infty$ and
\begin{equation}\label{e7-lm2}e^{\gamma_0 t}[(\BS(t)-\wdt S(t))^2+(\BP(t)-\wdt P(t))^2] \leq \frac{\wdt C_{0} ((\bar x-\wdt x)^2+(\bar y-\wdt y)^2+\delta^2)}\eps.
\end{equation}
Thus, we have from \eqref{lnwdtX} and \eqref{e7-lm2} that
\begin{equation}\label{e8-lm2}
\lim_{t\to\infty}[\wdt X(t)+|\wdt S(t)-\BS(t)|+|\wdt P(t)-\BP(t)|]=0.
\end{equation}
\begin{equation}\label{e9-lm2}
	\begin{aligned}
		\ln\wdt X(t)=&\ln\wdt x+\int_0^t \left(\mu_2(\xi(s))\wdt S(s)-h_1(\xi(s))\wdt P(s)-D(\xi(s))-\frac{\sigma_2^2(\xi(s))}2\right)ds
		+\int_0^t\sigma_2(\xi(s))W_2(s)\\
		=& \ln\wdt x+\int_0^t \left(\mu_2(\xi(t))\BS(t)-h_1(\xi(t))\BP(t)-D(\xi(t))-\frac{\sigma_2^2(\xi(t))}2\right) dt+\int_0^t\sigma_2(\xi(s))W_2(s)\\
		&+\int_0^t \left[\mu_2(\xi(t))[\BS(s)-\wdt S(s)]-h_1(\xi(t)[\BP(s)-\wdt P(s)]\right]ds
	\end{aligned}
\end{equation}
Diving both sides of \eqref{e9-lm2} by $t$ and then letting $t\to \infty$, we have from \eqref{e8-lm2} and the ergodicity of $(\BS(t),\BP(t),\xi(t))$ that
$$
\lim_{t\to\infty}\frac{\ln\wdt X(t)}t=\Lambda<0 \text{ for almost all }\omega\in\bigcap_{i=1}^7\Omega_i.
$$
Finally, because $\hat S(t)$ is the solution to \eqref{main} with initial condition $\widetilde\by$ under $\Q$ and $\Q_{s,p,\wdt\by,k}(\bigcap_{i=1}^7\Omega_i)\geq 1-10\eps$, we can claim that
$$
\PP_{\wdt\by}\left\{\lim_{t\to\infty}\frac{\ln X(t)}t=\Lambda<0\right\}=\Q_{s,p,\wdt\by,k}\left\{\lim_{t\to\infty}\frac{\ln \widetilde X(t)}t=\Lambda<0\right\}\geq 1-11\eps
$$
as long as $(\wdt s-s)^2+(\wdt p-p)^2+\wdt x^2\leq\varsigma^2$.
\end{proof}
\begin{proof}[Proof of Theorem \ref{thm2}.]
	Proposition \ref{lm2} shows that there is no invariant measure on $\R^{3,\circ}_+$.
	So $\bmu_b$ is the unique ergodic invariant probability measure of $(\BY(t),\xi(t))$ on $\R^3_+\times\M$.
	In view of Theorem \ref{thm1}, the family
	$\left\{ \Pi_t^{\by,k}(\cdot):=\dfrac1t\int_0^t\PP_{\by}\left\{(\BY(s),\xi(s))\in\cdot\right\}ds, t\geq 0\right\}$
	is tight in $\R^3_+\times\M$ and
	 any weak-limit of $\Pi_t^{\by,k}$ as $t\to\infty$ must be an invariant probability measure of $(\BY(t),\xi(t))$, that is, the weak-limit must be the unique invariant measure $\bmu_b$; we can refer to \cite{ethier2009markov, HN16, B23} for the proof that limit points of $\Pi_t^{\by,k}$ have to be invariant measures of $(\BY(t),\xi(t))$.
	 
	Let $R_\eps>0$ such that $\mu_{12}([R_{\eps}^{-1},R_{\eps}]^2)\geq 1-\eps.$
	By the Heine–Borel covering theorem, there exists $(x_1,y_1),\cdots, (x_l,y_l)$ such that
	$[R_{\eps}^{-1},R_{\eps}]^2$ is covered by the union of disks centered at $(x_k,y_k)$ with radius $\frac12\varsigma_{x_k,y_k,\eps}$, $k=1,\cdots,n$. Here $\varsigma$ is determined as in Proposition \ref{lm2}.
	Then, for any $\wdt\by\in [R_{\eps}^{-1},R_{\eps}]^2\times(0,\frac12\varsigma_m)$ with $\varsigma_m=\min_{k=1,\cdots,l}\{\varsigma_{x_k,y_k,\eps}\}$, there exists
	$k_{\wdt\by}\in\{1,\cdots,l\}$ such that
	$$(\wdt s-s_{k_{\wdt\by}})^2+(\wdt p-p_{k_{\wdt\by}})^2+\wdt x^2\leq\varsigma_m^2$$
	Thus, \begin{equation}\label{e3-thm3}
	\PP_{\wdt\by}\left\{\lim_{t\to\infty}\frac{\ln (t)}t=\Lambda<0\right\}>1-\eps, \wdt\by\in [R_{\eps}^{-1},R_{\eps}]^2\times(0,\varsigma_m).
	\end{equation}
For $\eps>0$ let $R_{\eps}>0$ be sufficiently large that $\bmu_{b}([R_{\eps}^{-1},R_{\eps}]^2\times\M)>1-\eps$. Because of the weak convergence of $\Pi_t^{\by,k}$ to $\bmu_b$, there exists a  $\check T=\check T(\by,\eps)>0$ such that
	$$\check\Pi^{\check T}_{\by}([R_{\eps}^{-1},R_{\eps}]^2\times(0,\varsigma_m))>1-\eps$$
	or in other words,
	$$\dfrac1{\check T}\int_0^{\check T}\PP_{\by}\{(\BY(t),\xi(t))\in ([R_{\eps}^{-1},R_{\eps}]^2\times(0,\varsigma_m))\times\M\}dt>1-\eps.$$
	Consequently,
	$$\PP_{\by}\{\hat\tau\leq\check T\}>1-\eps$$
	where $\hat\tau=\inf\{t\geq 0: \BY(t)\in [R_{\eps}^{-1},R_{\eps}]^2\times(0,\varsigma_m)\}$.
	Using the strong Markov property and \eqref{e3-thm3},
	we deduce that
	\begin{equation}
		\PP_{\by}\left\{ \lim_{t\to\infty}\frac{\ln X(t)}{t}=\Lambda\right\}\geq (1-\eps)(1-\eps)\geq 1-2\eps \text{ given } \by\in\R^{3,\circ}_+.
	\end{equation}
	Letting $\eps\to0$ we obtain the desired result.
\end{proof}
\section{Proof of Theorem \ref{thm3}}\label{s:pers}
Theorem \ref{thm3} is proved by arguments developed in \cite{benaim2022stochastic}.
We begin with citing \cite[Lemma 4.6]{benaim2022stochastic}.
\begin{lm}\label{lm3.0}
	Let $1<p\leq 2$. There exists $c_p>0$ such that for any $a>0$ and $x\in\R$ we have
	\beq\label{lm3.0-e1}
	|a+x|^{p}\leq
		a^{p}+pa^{p-1}x+c_p|x|^{p}
	\eeq
	Moreover,  there exists $d_{p, b}>0$ depending only on $p,b>0$ such that if $x + a \geq 0$ then
\beq\label{lm3.0-e3}
(a+x)^{p}-b(a+x)^{p-1} \leq
a^{p}+ pa^{p-1} x-\frac{b}2 a^{p-1}+c_{b,p}(|x|^{p}+1)
\eeq
It follows straightforwardly from \eqref{lm3.0-e1} that
	for a random variable $R$ and a constant $c>0$, one has that there exists $\tilde K_c >0$ such that
	\beq\label{lm3.0-e2}
	\E |R+c|^{p}\leq
		c^{p}+pc^{p-1}{\E Y}+\tilde K_c\E |R|^{1+p} .
	\eeq
	\end{lm}
In this section, let
$\gamma_2>0$, $\gamma_3>0$ such that
$$
L(\gamma_2\vee\gamma_3)\leq \frac12\min\{\alpha_1,\alpha_4-\alpha_2,\alpha _5- \alpha_3\}
\text{ and }
\gamma_2\Lambda-\gamma_3 \left(\alpha_3+\frac{\sigma_3^2}2\right)>0 $$
and pick
$$\rho=\frac12\left[\left(-\gamma_3\lambda_2\right)\vee \left(\gamma_2\lambda_1-\gamma_3\left(\alpha_3+\frac{\sigma_3^2}2\right) \right)\right]>0.$$
Let $\beta>0$ such that $\beta\max\{h_1(k),\mu_2(k): k\in\M\}\leq\frac{\alpha_0}2$.
Pick $c_1>0$ such that  $(s+\alpha x+p) -\beta \ln x+c_1\geq 0$ for any $(s,x,p)\in\R^{3,\circ}_+$
and 
$$
V(\by)=U(\by)-\beta\ln x+c_1=s+\alpha x+p -\beta \ln x+c_1, \by\in\R^{3,\circ}_+.
$$
Using \eqref{LUP} we get the following estimate
\begin{equation}\label{LV}
	\begin{aligned}
[\op V](\by,k)=&[\op U](\by,k) -\beta\left(\mu_2(k)s-h_1(k)p-\frac{\sigma_2^2(k)}2\right) \\
\leq& C_1-\alpha_0 (s+\alpha x+p) -\beta\left(\mu_2(k)s-h_1(k)p-\frac{\sigma_2^2(k)}2\right)\\
\leq & C_1-\frac{\alpha_0}2(s+\alpha x+p)\\
\leq & C_1\1_{\{U(\by)\leq M\}}-\frac{\alpha_0}4 U(\by)
\end{aligned}
\end{equation}
where $M=\frac{4C_1}{\alpha_0}$.
\begin{lm}\label{lm3.1}
	There exist $T^\backprime>0, \delta>0$ such that
	$$
	\E_{\by,k}\int_0^T \op V(\BY(s))ds \leq -\frac{\beta\lambda}2 T $$
	for any  $T\in [T^\backprime, n^\backprime T^\backprime]$, $\by\in\R^{3,\circ}_+$, $|\by|\leq M$, and $\dist(\by,\partial\R^{3,\circ}_+)\leq\delta$.
	
\end{lm}
\begin{proof}
	On the boundary $\R^2_+\times\{0\}\times\M$, $\bmu_b$ is the unique invariant probability measure..
	In view of Theorem \ref{thm1} and \cite[Lemma 3.3, Lemma 3.4, Lemma 3.5]{HN16},  we easily deduce that
	
	\begin{enumerate}
		\item[(C1)]	$(x+y+  z)^{q_0}$ is $\bmu_b$-integrable
		and
\begin{equation}\label{e0-lm3.1}
	\begin{aligned}
	\sum_{k\in\M}	&\int_{\R^3_+} [\op U](\by,k)\bmu_b(d\by,k)\\
	=&\sum_{k\in\M}	\int_{\R^3_+} \left[D(k)(s-\alpha x+y) + (\alpha \mu_2(k)-\mu_1(k))sx-(\alpha h_1(k)+h_2(k))xp\right]\bmu_b(d\by,k)=0
	\end{aligned}
\end{equation}
		\\
		\item[(C2)] $
		\{\check\Pi_t^{\by,k}: t\geq 1, |\by|\leq M\}
		$ is tight and all its weak-limit as $t\to\infty$ must be an invariant measure of $(\BY(t),\xi(t))$.
		\item[(C3)]
		For a sequence of bounded initial points $\{\by_k\in\R^3_+\}$ and an increasing sequence $T_k\to\infty$ as $k\to\infty$,
		if $\{\check\Pi_{T_k}^{\by_k}\}$ converges to $\mu$ as $T_k$ tends to $\infty$ then
		$$
		\lim_{k\to\infty} \int_{\R^3_+}h(\by)\check\Pi_{T_k}^{\by_k}(d\by)=\int_{\R^3_+}h(\by)\mu(d\by)
		$$
		for any continuous function $h(\by)$ satisfying $h(\by)\leq C_h(1+x+y)^{q}$ for some $C_h>0, 0<q<q_0$. 
	\end{enumerate}
Next, combining \eqref{LV}, \eqref{e0-lm3.1}, we obtain
\begin{equation}\label{e3-lm3.1}
\sum_{k\in\M}\int_{\R^3_+}[\op V](\by,k)\bmu_b(\by,k)=\sum_{k\in\M}\int_{\R^3_+}\left[[\op V](\by,k)+\beta\left(\mu_2(k)s-h_1(k)p-\frac{\sigma_2^2(k)}2\right)\right]\bmu_b(\by,k)=-\beta\lambda.
\end{equation}
Now, we show the existence of $T^\backprime=T^\backprime(M)>0$ such that if $\by\in\R^2_+\times\{0\}$ and $|\hat\by|\leq M$ then
\begin{equation}\label{e2-lm3.1}
\E_{\hat \by, k} \frac{1}{T}\int_0^T \op V(\BY(s))ds= \int_{\R^3_+}\op V(\by)\check\Pi_{T}^{\hat \by, k}(d\by)\leq -\frac34\beta\lambda.
\end{equation}
If it is untrue, there exists a sequence $\{\by_n,k_n\}\subset \R^2_+\times\{0\}\times\M$ such that $|\by_n|\leq M$ and a sequence $T_n\uparrow \infty$ such that
$$
\E_{\by_n,k_n}\frac1{T_n}\int_0^{T_n} \op V(X(s))ds>-\frac34\beta\lambda.
$$
In view of Claim (C2), there is a subsequence, still denoted by $\{\by_n,k_n\}$ and $\{T_n\}$ for convenience, such that
$\check\Pi_{T_n}^{\by_n, k_n}$ converges to an invariant probability measure $\bmu$ as $n\to\infty$.
Since $\bmu_b$ is the unique invariant measure on $\R^2_+\times\{0\}\times\M$, we deduce from Claim (C3) that
$$
\int_{\R^3_+}\op V(\by)\bmu_b(d\by)=\lim_{n\to\infty}\int_{\R^3_+}\op V(\by)\check\Pi_{T_n}^{\by_n,k}(d\by)\geq-\frac34\beta\lambda
$$
which contradicts \eqref{e3-lm3.1}. As a result, \eqref{e2-lm3.1} holds true.

Then, we deduce from the Feller-Markov property of $(\BY(t),\xi(t))$ and the uniform moment boundedness \eqref{e1-thm1} that there exists $\delta>0$ such that
$$\E_{\by,k}\int_0^T \op V(X(s))ds\leq -\frac{\beta\lambda}2 T, T\in[T^\backprime, n^\backprime T^\backprime]$$
for any $\by\in\R^3_+$, $U(\by)\leq M$.
\end{proof}

\begin{prop}\label{lm3.3}
	Let $q$ be any number in the interval $(1,q_0)$, $n^\backprime>0$ such that
	$(n^\backprime -1)\frac{\alpha_0}4-2^{q-1}C_1\geq \frac{\beta\lambda}4$,
	and $U(\by)=1+x+y+  z$.
	There are $\kappa^\backprime>0$ and $C, C^\backprime >0$ independent of $\by$ such that
	$$
	\E_\by [C|\BY(t)|^q+V^q(\BY(n^\backprime T^\backprime))]\leq CU^q(x)+V^q(x)-\kappa^\backprime [CU^q(\by)+V^q(\by)]^{\frac{q-1}q} + C^\backprime
	$$
\end{prop}
\begin{proof}
	First we assume that $1<q\leq 2$.
	In the sequel, $C^\backprime$ is a generic constant depending on $T^\backprime, M, n^\backprime$ but independent of $(\by,k)\in\R^3_+\times\M$.
	$C^\backprime$ can differ from line to line.
	Suppose $\BY(0)=\by$. We have from It\^o's formula that
	$$
	V(\BY(t))=V(\by)+\int_0^t \op V(\BY(s))ds+\wdt M(t),
	$$
	where
	$$
	\wdt M(t):=\int_0^t (\sigma_1(\xi(s)) S(s)dW_1(s)+\alpha\sigma_2(\xi(s)) X(s)dW_2(t)+ \sigma_3(\xi(s)) P(s)dW_3(s)-\beta\sigma_2(\xi(s))dW_2(s))
	$$
	 is a martingale with quadratic variation:
	\beq\label{e3-lm3.3}\langle \wdt M(t)\rangle=\int_0^t \left(\sigma_1^2(\xi(s))S^2(s)+\alpha^2\sigma_2^2(\xi(s)) (X(s)-\beta)^2+\sigma_3^2(\xi(s))P^2(s)\right)ds\leq K \int_0^t U^2(\BY(s))ds,
	\eeq
	for some constant $K=K(\sigma_1,\sigma_2,\sigma_3,\gamma_2,\gamma_3)$.
	
	Note that $\op V(\by)\leq C_1$. Thus $$V(X(T))=V(x)+\int_0^T \op V(X(s))ds+\wdt M(T)\leq V(x)+C_1 T+\wdt M(T).$$  Applying \eqref{lm3.0-e2} to the inequality above yields
	\beq\label{e6-lm3.3}
	\begin{aligned}
		\E_\by [V(\BP(T))]^q
		\leq &  V^q(\by) + qC_1 T V^{q-1}(\by)+ C^\backprime U^q(\by), \quad T\leq n^\backprime T^\backprime.
	\end{aligned}
	\eeq

		%
	On the other hand, since $|\op V(\by)|\leq K_0 (|\by|+1),\forall \by\in\R^3_+$ for some constant $K_0$, we deduce from It\^o's isometry and Hölder's inequality  that
	\beq\label{e4-lm3.3}
	\E_\by\left|\int_0^t LV(\BY(s))ds\right|^q+\E_\by \left|M(t)\right|^q \leq C^\backprime U^q(\by) \quad t\leq n^\backprime T^\backprime,\by\in\R^{3,\circ}_+.
	\eeq
	It follows from \eqref{e4-lm3.3} and \eqref{lm3.0-e2} that
	\beq\label{e0-lm3.3}
	\begin{aligned}
		\E_\by [V(\BY(t))]^q\leq&  V^q(\by) + q\left[\E_\by \int_0^t \op V(\BY(s))ds\right] V^{q-1}(\by)+ C^\backprime\E_\by\left|\int_0^t \op V(\BY(s))ds+M(t)\right|^q\\
		\leq &  V^q(\by) + q\left[\E_\by \int_0^t \op V(\BY(s))ds \right] V^{q-1}(\by)+ C^\backprime U^q(\by), \quad t\leq n^\backprime T^\backprime .
	\end{aligned}
	\eeq
	
Recall that if $U(\by)\leq M $ and $ \dist(\by,\partial \R_+^3)\leq \delta$, then  $\E_\by \int_0^t \op V(\BY(s))ds\leq -\frac{\beta\lambda}2 t$, $t\in[T^\backprime , n^\backprime T^\backprime ]$. As a result, for $T\in[T^\backprime , n^\backprime T^\backprime]$, if $U(\by)\leq M$,
	\beq\label{e5-lm3.3}
	\begin{aligned}
		\E_\by [V(\BY(T))]^q
		\leq &  V^q(\by) - q\frac{\beta\lambda}2 T V^{q-1}(\by)+ C^\backprime U^q(\by)\\
		\leq & V^q(\by) - q\frac{\beta\lambda}2 T V^{q-1}(\by)+ C^\backprime.
	\end{aligned}
	\eeq
%
		Define 	 $$\zeta=\inf\{t\geq 0: U(\BY(t))\leq M\}\wedge (n^\backprime T^\backprime).$$
For $t\leq \zeta$, we have from \eqref{LV} that
	\beq\label{e9-lm3.3}
	V(\BY(t))= V(\by)+\int_0^tLV(\BY(s))ds +\wdt M(t)\leq V(\by)- \frac{\alpha_0}4t+\wdt M(t).
	\eeq
	Using \eqref{e5-lm3.3} and the strong Markov property of $\{\BY(t),\xi(t)\}$, we can estimate
	\beq\label{e1-lm3.3}
	\begin{aligned}
		\E_x& \left[\1_{\{\zeta \leq T^\backprime (n^\backprime -1)\}} V^q(\BY(n^\backprime T^\backprime ))\right]\\[0.5ex]
		\leq &\E_x \left[\1_{\{\zeta \leq T^\backprime (n^\backprime -1)\}} \left[V^q(\BY(\zeta))+C^\backprime \right] \right]\\[0.5ex]
		&-\E_x\left[\1_{\{\zeta \leq T^\backprime (n^\backprime -1)\}}q\frac{\beta\lambda}2(n^\backprime T^\backprime -\zeta)V^{q-1}(\BY(\zeta))\right]\\[0.5ex]
		\leq& \E_x \left[\1_{\{\zeta \leq T^\backprime (n^\backprime -1)\}} (V(\by)+\wdt M(\zeta))^q+C^\backprime  \right] \\[0.5ex]
		&-q\frac{\beta\lambda}2 T^\backprime \E_x\left[\1_{\{\zeta \leq T^\backprime (n^\backprime -1)\}}(V(\by)+\wdt M(\zeta))^{q-1}\right]\\[0.5ex]
		&\text{ (due to \eqref{e9-lm3.3})}\\
		\leq& \E_x\left[\1_{\{\zeta \leq T^\backprime (n^\backprime -1)\}}\left( V^q(\by)-\frac{q\frac{\beta\lambda}2 T^\backprime }2 V^{q-1}(\by)+ q\wdt M(\zeta)V^{q-1}(\by)+C^\backprime (|\wdt M(\zeta)|^q+1)\right)\right]\\[0.5ex]
		&\text{ ( thanks to \eqref{lm3.0-e3} ).}
	\end{aligned}
	\eeq
	If $T^\backprime (n^\backprime -1)\leq \zeta \leq T^\backprime n^\backprime $ we have
	\begingroup
	\allowdisplaybreaks
	\begin{align*}
		\E_x &\left[\1_{\{\zeta \geq T^\backprime (n^\backprime -1)\}} V^q(\BY(n^\backprime T^\backprime ))\right]\\[0.5ex]
		\leq &\E_x \left[\1_{\{\zeta \geq T^\backprime (n^\backprime -1)\}} V^q(\BY(\zeta))+C^\backprime \right] \\[0.5ex]
		&+qC_1\E_x\left[\1_{\{\zeta \geq T^\backprime (n^\backprime -1)\}}(n^\backprime T^\backprime -\zeta)V^{q-1}(\BY(\zeta))\right] \\[0.5ex]
		&\text{ (by applying \eqref{e6-lm3.3} and the strong Markov property)}\\[0.5ex]
		\leq& \E_x \left[\1_{\{\zeta \geq T^\backprime (n^\backprime -1)\}} [(V(\by)+\wdt M(\zeta)- \frac{\alpha_0}4\zeta)^q+C^\backprime ]\right] \\[0.5ex]
		&+qC_1 T^\backprime \E_x\left[\1_{\{\zeta \geq T^\backprime (n^\backprime -1)\}}(V(\by)+\wdt M(\zeta)- \frac{\alpha_0}4\zeta)^{q-1}\right] \\[0.5ex]
		&\text{ (in view of \eqref{e9-lm3.3})}\\[0.5ex]
		\leq& \E_x\left[\1_{\{\zeta \geq T^\backprime (n^\backprime -1)\}}\left( V^q(\by)-q \frac{\alpha_0}4\zeta V^{q-1}(\by)+q\wdt M(\zeta) V^{q-1}(\by) + C^\backprime \left(|\wdt M(\zeta)|+1\right)^q\right)\right]\\[0.5ex]
		&+ 2qC_1 T^\backprime \E_x\left[\1_{\{\zeta \geq T^\backprime (n^\backprime -1)\}} \left( V^{q-1}(\by)+|\wdt M(\zeta)|^{q-1}\right)\right]\\[0.5ex]
		& \text{ (using \eqref{lm3.0-e1} and the inequality $|x+y|^{q-1}\leq 2(|x|^{q-1}+|y|^{q-1})$)} \\[0.5ex]
		\leq & \E_x \left[ \1_{\{\zeta \geq T^\backprime (n^\backprime -1)\}} \left( V^q(\by)-\frac{q\beta\lambda T^\backprime }4 V^{q-1}(\by)+q \wdt M(\zeta)V^{q-1}(\by) +C^\backprime \left(|\wdt M(\zeta)|+1\right)^q \right)\right] \\
		& \text{ (because } \frac{\alpha_0}4 \zeta\geq \frac{\alpha_0}4T^\backprime (n^\backprime -1))\geq (2C_1+\frac{\beta\lambda}4)T^\backprime).\\[-2ex]
		& \stepcounter{equation}\tag{\theequation}\label{e2-lm3.3}
	\end{align*}
	\endgroup
	As a result, by adding \eqref{e1-lm3.3} and \eqref{e2-lm3.3} and noting that $\E_x \wdt M(\zeta)=0$, we have
	\beq\label{e10-lm3.3}
	\begin{aligned}
		\E_x V^q(\BY(n^\backprime T^\backprime ))\leq&  V^q(\by)-q\frac{\beta\lambda}4 T^\backprime  V^{q-1}(\by)+ C^\backprime  \E_x (|\wdt M(\zeta)|+1)^q\\
		\leq& V^q(\by)-q\frac{\beta\lambda}4 T^\backprime  V^{q-1}(\by) + C^\backprime  U^{q}(\by)
	\end{aligned}
	\eeq
	where the inequality $\E_x (|\wdt M(\zeta)|+1)^q\leq C^\backprime  U^{q}(\by)$ comes from an application of
	the Burkholder-Davis-Gundy Inequality, Hölder's inequality and \eqref{e3-lm3.3} and \eqref{e2-thm1}.
	
	From \eqref{e1-thm1}, we have
	\beq\label{e11-lm3.3}
	\begin{aligned}
		\E_x U^q(\BY(n^\backprime T^\backprime ))\leq&  U^q(\by)-\left(1-e^{-k_{2q} n^\backprime T^\backprime }\right) U^q(\by) + \frac{k_{1q}}{k_{2q}}.
	\end{aligned}
	\eeq
	Combining \eqref{e10-lm3.3} and \eqref{e11-lm3.3}, we can easily get that
	\beq\label{e12-lm3.3}
	\E_x \left[V^q(\BY(n^\backprime T^\backprime ))+CU^q(\BY(n^\backprime T^\backprime ))\right]
	\leq V^q(\by)+C U^q(\by)-\kappa^\backprime  [V^q(\by)+CU^q(\by)]^{(q-1)/q} + C^\backprime ,
	\eeq
	for some $ \kappa^\backprime >0, C^\backprime >0$ and sufficiently large $C$.
\end{proof}

\section{Nonlinear Perturbed models}\label{s:nonlinear}
As a novelty, we also wanted to explore what happens when the white noise perturbations which influence \eqref{main_det} are nonlinear.
More specifically, assume that the dynamics is given by
\begin{equation}\label{main4}
	\begin{aligned}
		dS(t) &= [D(S_0-S(t))-\mu_1(\xi(t))S(t)X(t)]dt+(\sigma_1(\xi(t))+\sigma_4(\xi(t))S(t)) S(t)dW_1(t)\\
		dX(t)&=[\mu_2(\xi(t))S(t)X(t)-D(\xi(t))X(t)-h_2(\xi(t))X(t)P(t)]dt+(\sigma_2(\xi(t))+\sigma_5(\xi(t)X(t))X(t)dW_2(t)\\
		dP(t)&=[D(\xi(t))(P_0-P(t))-h_3(\xi(t))X(t)P(t)]dt+(\sigma_3(\xi(t))+\sigma_6(\xi(t))P(t))P(t)dW_3(t)
	\end{aligned}
\end{equation}
where $(W_1(t), W_2(t), W_3(t))$ is a standard Brownian motion on $\R^3$.
Let $\BY(t):=(S(t), X(t), P(t))$ and let $\by \in \R_{+}^{3,\circ}$ denote the initial conditions, that is $\BY(0):=(S(0), X(0), P(0))=\by$. We will use the framework developed in \cite{HNC20} which allows to treat processes and auxiliary variables. This way, the process can be in Kolmogorov form while the auxiliary variables can be in non-Kolmogorov form. 

Specifically, we assume that $(S(t), P(t),\alpha(t))$ is the auxiliary variable and we focus on the extinction and persistence of the process $X(t)$.
Letting $\wdt U(\by)= (s+x+p)^{\frac12}$
we have
\begin{equation}\label{lwdtU}
[\op_2 \wdt U](\by,k)\leq \wdt c_1 - \wdt c_2\wdt U^{3}(\by).
\end{equation}
We can therefore see that \cite[Assumption 3.1]{HNC20} holds for the model \ref{main4}.

When $X(t)=0\,t\geq0$, \eqref{main4} reduces to the system:
\begin{equation}\label{main5}
	\begin{aligned}
		d\BS(t) &= [D(S_0-\BS(t))]dt+(\sigma_1(\xi(t))+\sigma_4(\xi(t))\BS(t)) \BS(t)dW_1(t)\\
		d\BP(t)&=D(\xi(t))(P_0-\BP(t))dt+(\sigma_3(\xi(t))+\sigma_6(\xi(t))\BP(t))\BP(t)dW_3(t)
	\end{aligned}
\end{equation}
In view of \eqref{lwdtU}, by applying Dynkin's formula in the same manner as in the proof of Theorem \ref{thm1}, we have 
$$
\E_{\by,k} \wdt U(s,0,p)\leq e^{-\wdt c_2 t} \tilde U(s,0,p)+ \frac{\wdt c_1}{\wdt c_2}, t\geq 0,
$$
and 
$$
\E_{\by,k}\frac1t\int_0^t \wdt U^5(\BS(u),0,P(u))du\leq \wdt c_1 + \frac1t\wdt U(s,0,p).
$$
As a result, the set of invariant probability measures of $(\BS(t),\BP(t),\xi(t))$ is non empty and for any invariant probability measure $\bmu$
$$
\sum_{k\in\M}\int_{\R^2_+} U^5(s,0,p)\bmu(s,0,p,k)\leq\wdt c_1.
$$
Since $ U^5(s,0,p)=(s+p)^{\frac52}$ is $\bmu$-integrable, we can define
$$
\lambda(\bmu)=\sum_{k\in\M}\int_{\R^2_+} \left(\mu_2(k)s-h_2(k)p-\frac{\sigma_2^2(k)}2\right)\bmu(dsdp,k).
$$
While the probability measure $\bmu$ is probably not unique, the value $\lambda(\bmu)$ is unique due to the following claim.
\begin{lm} For any invariant probability measure $\bmu$ of $(\BS(t),\BP(t),\xi(t))$, we have 
$$	\lambda(\bmu)=\sum_{k\in\M}\left[D(k)(\xi(k)+\zeta(k))-\frac{\sigma_2^2(k)}2\right]\pi_k.$$
\end{lm}
\begin{proof}
	The proof is almost identical to Lemma \ref{lm0} and is therefore omitted.
\end{proof}
Because \cite[Assumption 3.1]{HNC20} holds for model \ref{main4} and $\lambda(\bmu)$ is unique, we can apply Theorem 3.2, Theorem 3.4 and Theorem 3.5 of \cite{HNC20} to obtain the following result.

\begin{thm}
	If $\lambda(\bmu)>0$ then $X(t)$ is stochastically persistent in probability, that is
	for any $\eps>0$, there exists $\Delta>0$ such that
	$$
	\liminf_{t\to\infty}\PP_{\by,k}\{X(t)\geq\Delta\}\geq 1-\eps, (\by,k)\in\R^{3,\circ}_+\times\M.
	$$
	We also have that
	$$
	\PP_{\by,k}\left\{\liminf_{t\to\infty}\frac1t\int_0^tX(s)ds>0\right\}=1.
	$$
	If $\lambda(\bmu)<0$ then for any $\eps>0$, $(s,p)\in\R^2_+$, there exists $\delta_{s,p}>0$ such that
		$$
	\PP_{\wdt\by,k}\left\{\lim{t\to\infty}\frac1t
	\ln X(t)=\lambda(\bmu)<0\right\}\geq 1-\eps
	$$
	whenever $\wdt\by=(\wdt s,\wdt x,\wdt p)$ satisfies $|\wdt s-s|+|\wdt p-p|+\wdt x\leq\delta_{s,p}$.
	Moreover, if the boundary $\R^2_+\times\{0\}$ is accessible from the interior, that is, for any $\by\in\R^3_+, k\in \M$, we can find $K>0$ such that
$$
	\sup_{t\geq0}\PP_{\by,k}\{|S(t)|+|P(t)|\leq K \text{ and }X(t)<\eps\}>0  \text{ for any }\eps>0,
	$$
	then 
	$$
	\PP_{\by,k}\left\{\lim_{t\to\infty}\frac1t
	\ln X(t)=\lambda(\bmu)<0\right\}=1, \text{ for any }(\by,k)\in\R^{3,\circ}_+\times\M.
	$$
\end{thm}
The next result gives us an easy condition which ensures almost sure extinction.
\begin{cor}
	Suppose that $\lambda(\bmu)<0$. If there is $k_0\in\M$ such that $\sigma_2(k_0)\sigma_5(k_0)>0$ then we have
	$$
	\PP_{\by,k}\left\{\lim_{t\to\infty}\frac1t
	\ln X(t)=\lambda(\bmu)<0\right\}=1, \text{ for any }(\by,k)\in\R^{3,\circ}_+\times\M.
	$$
\end{cor}
\begin{proof}
	We only need to show that if $\sigma_2(k_0)\sigma_5(k_0)>0$ then the boundary $\R^2_+\times\{0\}$ is accessible from the interior. It turns out that the accessibility of a certain region of space can be checked using deterministic control theory.
Consider the following control system associated with \eqref{main4}:
\begin{equation}\label{control1}
	\begin{aligned}
		dS_u(t) &= [D(S_0-S_u(t))-\mu_1(\psi(t))S_u(t)X_u(t)+ \frac12(\sigma_1(\psi(t))+2\sigma_4(\psi(t))S_u(t))]dt \\&+ (\sigma_1(\psi(t))+\sigma_4(\psi(t))S_u(t)) S_u(t)\phi_1(t)dt\\
		dX_u(t)&=[\mu_2(\psi(t))S_u(t)X_u(t)-D(\psi(t))X_u(t)-h_2(\psi(t))X_u(t)P_u(t)+\dfrac12(\sigma_2(\psi(t))+2\sigma_5(\psi(t)X_u(t))]dt\\&+(\sigma_2(\psi(t))+\sigma_5(\psi(t)X_u(t))X_u(t)\phi_2(t)dt\\
		dP_u(t)&=[D(\psi(t))(P_0-P_u(t))-h_3(\psi(t))X_u(t)P_u(t)\\&+\frac12(\sigma_3(\psi(t))+2\sigma_6(\psi(t))P_u(t)]dt+(\sigma_3(\xi(t))+\sigma_6(\xi(t))P_u(t))P_u(t)\phi_1(t)dt
	\end{aligned}
\end{equation}
where $\phi_i(t), i=1,2,3$ are real-valued c\`adl\`ag functions and $\psi(t)$ a $\M$-valued  c\`adl\`ag function.
 In view of various support theorems, see Sections 4.3, 5.1 and 6 from and \cite{B23}, if we can find controls $u(t)=(\psi(t),\phi_1(t),\phi_2(t),\phi_3(t))$ satisfying \eqref{control1} such that
$$
\psi(0)=k, (S_u(0),X_u(0),P_u(0)=\by\quad \text{ and }
(S_u(T),X_u(T),P_u(T))=\by^\diamond=(s^\diamond, x^\diamond, p^\diamond)
$$
then \begin{equation}\label{support}\PP_{\by,k}\left\{|\BY(T)-\by^\diamond|\leq\eps \right\}>0 \text{ for any } \eps>0.\end{equation}

Now, we define a control as follows: $\psi(t)\equiv k_0, t\geq1$,
$\phi_1(t)\equiv -\sgn(\sigma_4(k_0)), \phi_3(t)\equiv -\sgn(\sigma_6(k_0))$
and $\phi_2(t)$ is a feedback control such that $d X_u(t)\leq - X_u(t)dt, t\geq 1$.
To be more specific, since $\sigma_2(k_0)\sigma_5(k_0)>0$, which leads to $\sgn(\sigma_2(k_0))\left[\sigma_2(k_0)+\sigma_5(k_0)X_u(t)\right]>|\sigma_2(k_0)|$, we can select
$$\phi_2(t)=
\begin{cases}
-1- \mu_2(k_0)S_u(t)-\frac12(\sigma_2(k_0)+2\sigma_5(k_0)X_u(t)) &\text{ if } \sigma_2(k_0)>0
\\
1+ \mu_2(k_0)S_u(t)& \text{ if }\sigma_2(k_0)<0.
\end{cases}
$$
Since $d X_u(t)\leq - X_u(t)dt, t\geq 1$, we have 
\begin{equation}\label{X_u}\lim_{t\to\infty} X_u(t)=0
	\end{equation}
On the other hand, with $\psi(t)\equiv k_0, t\geq1$,
$\phi_1(t)=\phi_2(t)=0, t\in[0,1)$ and $\phi_1(t)= -\sgn(\sigma_4(k_0)), \phi_3(t)= -\sgn(\sigma_6(k_0)), t\geq 1$, one can see that
$$
d[S_u(t)+P_u(t)]\leq [k_1 - k_2(S_u(t)+P_u(t))^2]dt, t\geq 1.
$$
This inequality implies
\begin{equation}\label{SP_u}
\limsup_{t\to\infty} (S_u(t)+P_u(t))\leq k_3:=\left(\frac{k_1}{k_2}\right)^{\frac12}.
\end{equation}
Combining \eqref{X_u} and \eqref{SP_u} and applying  \eqref{support}, we deduce that
$$
\sup_{t\geq0}\PP_{\by,k}\{|S(t)|+|P(t)|\leq k_3+1 \text{ and }X(t)<\eps\}>0  \text{ for any }\eps>0,
$$
which means the boundary $\R^2_+\times\{0\}$ is accessible.
The corollary is therefore proved.
\end{proof}
\textbf{Acknowledgments:} The research was supported by the research project QG.22.10 ``Asymptotic behaviour of mathematical models in ecology" of Vietnam National University, Hanoi for Nguyen Trong Hieu. A. Hening, D. Nguyen, and N. Nguyen acknowledge
support from the NSF through the grants DMS CAREER 2339000 and DMS-2407669.

\bibliographystyle{amsalpha}
\bibliography{LV}
\end{document}